\documentclass[english]{article}
\usepackage[T1]{fontenc}
\usepackage[latin9]{inputenc}
\usepackage{babel}
\usepackage{array}
\usepackage{float}
\usepackage{booktabs}
\usepackage{units}
\usepackage{amsthm}
\usepackage{amsmath}
\usepackage{amssymb}
\usepackage{wasysym}
\usepackage{graphicx}
\usepackage{esint}
\usepackage[unicode=true,pdfusetitle,
 bookmarks=false,
 breaklinks=false,pdfborder={0 0 1},backref=false,colorlinks=false]
 {hyperref}

\makeatletter

\providecommand{\tabularnewline}{\\}

\numberwithin{equation}{section}
\numberwithin{figure}{section}
\newcommand{\lyxaddress}[1]{
\par {\raggedright #1
\vspace{1.4em}
\noindent\par}
}
\theoremstyle{plain}
\newtheorem{thm}{\protect\theoremname}
  \theoremstyle{plain}
  \newtheorem{prop}[thm]{\protect\propositionname}
  \theoremstyle{definition}
  \newtheorem{defn}[thm]{\protect\definitionname}
  \theoremstyle{remark}
  \newtheorem*{rem*}{\protect\remarkname}
  \theoremstyle{plain}
  \newtheorem{cor}[thm]{\protect\corollaryname}
  \theoremstyle{remark}
  \newtheorem{rem}[thm]{\protect\remarkname}

\date{}
\usepackage[leftcaption]{sidecap}
\sidecaptionvpos{figure}{c}
\newcommand{\FigBesBeg}[1][1.0]{%
 \let\MyFigure\figure
 \let\MyEndfigure\endfigure
 \renewenvironment{figure}[1]{\begin{SCfigure}[#1]##1}{\end{SCfigure}}}

\newcommand{\FigBesEnd}{%
 \let\figure\MyFigure
 \let\endfigure\MyEndfigure}

\DeclareMathOperator{\fix}{fix}
\DeclareMathOperator{\vol}{vol}
\DeclareMathOperator{\Hom}{Hom}
\DeclareMathOperator{\Spec}{Spec}

\usepackage{calrsfs}

\makeatother

  \providecommand{\corollaryname}{Corollary}
  \providecommand{\definitionname}{Definition}
  \providecommand{\propositionname}{Proposition}
  \providecommand{\remarkname}{Remark}
\providecommand{\theoremname}{Theorem}

\begin{document}

\title{On $G$-sets and Isospectrality}

\author{Ori Parzanchevski%
\thanks{Supported by an Advanced ERC Grant.%
}}

\maketitle

\lyxaddress{\begin{center}
parzan@math.huji.ac.il
\par\end{center}}
\begin{abstract}
We study finite $G$-sets and their tensor product with Riemannian
manifolds, and obtain results on isospectral quotients and covers.
In particular, we show the following: if $M$ is a compact connected
Riemannian manifold (or orbifold) whose fundamental group has a finite
non-cyclic quotient, then $M$ has isospectral non-isometric covers.
\end{abstract}

\section{\label{sec:Intro}Introduction}

Two Riemannian manifolds are said to be \emph{isospectral} if they
have the same spectrum of the Laplace operator (see Definition \ref{def:spectrum}).
The question whether isospectral manifolds are necessarily isometric
has gained popularity as {}``\textit{Can one hear the shape of a
drum?}\textit{\emph{'' \cite{kac1966can}, and it was answered negatively
for many classes of manifolds (e.g., \cite{milnor1964eigenvalues,buser1986isospectral,gordon1992one,conway1994some}).
}}In 1985, Sunada described a general group-theoretic method for constructing
isospectral Riemannian manifolds \cite{sunada1985riemannian}, and
recently this method was presented as a special case of a more general
one \cite{band2009isospectral,parzanchevski2010linear}. In this paper
we explore a broader special case of the latter theory, obtaining
the following, somewhat surprising, result (Corollary \ref{cor:non-cyclic-isospectral}):
\begin{quote}
Let $G$ be a finite non-cyclic group which acts faithfully on a compact
connected Riemannian manifold $M$. Then there exist $r\in\mathbb{N}$
and subgroups $H_{1},\ldots,H_{r}$ and $K_{1},\ldots,K_{r}$ of $G$
such that the disjoint unions $\bigcup_{i=1}^{r}\nicefrac{M}{H_{i}}$
and $\bigcup_{i=1}^{r}\nicefrac{M}{K_{i}}$ are isospectral non-isometric
manifolds (or orbifolds%
\footnote{If $G$ does not act freely on $M$ (i.e., some $g\in G\backslash\left\{ e\right\} $
acts on $M$ with fixed points), then $\bigcup\nicefrac{M}{H_{i}}$
and $\bigcup\nicefrac{M}{K_{i}}$ are in general orbifolds. A reader
not interested in orbifolds can assume that all spaces in the paper
are manifolds, at the cost of limiting the discussion to free actions.%
}).
\end{quote}
The result mentioned in the abstract follows immediately (Corollary
\ref{cor:non-cyclic-quotient}).

\medskip{}

Throughout this paper $M$ denotes a compact Riemannian manifold,
and $G$ a finite group which acts on it by isometries. In these settings,
Sunada's theorem \cite{sunada1985riemannian} states that if two subgroups
$H,K\leq G$ satisfy 
\begin{equation}
\forall g\in G:\quad\left|g^{G}\cap H\right|=\left|g^{G}\cap K\right|\label{eq:Sunada-condition}
\end{equation}
(where $g^{G}$ denotes the conjugacy class of $g$ in $G$), then
the quotients $\nicefrac{M}{H}$ and $\nicefrac{M}{K}$ are isospectral.
In fact, it is not harder to show (Corollary \ref{cor:isospectrality})
that if two collections $H_{1},\ldots,H_{r}$ and $K_{1},\ldots,K_{r}$
of subgroups of $G$ satisfy
\begin{equation}
\forall g\in G:\quad\sum_{i=1}^{r}\frac{\left|g^{G}\cap H_{i}\right|}{\left|H_{i}\right|}=\sum_{i=1}^{r}\frac{\left|g^{G}\cap K_{i}\right|}{\left|K_{i}\right|}\label{eq:GenSunada-condition}
\end{equation}
then $\bigcup\nicefrac{M}{H_{i}}$ and $\bigcup\nicefrac{M}{K_{i}}$
are isospectral%
\footnote{In this paper $\bigcup$ always stands for disjoint union.%
}. We shall see, however, that in contrast with Sunada pairs ($H,K$
satisfying \eqref{eq:Sunada-condition}), collections satisfying \eqref{eq:GenSunada-condition}
are rather abundant. In fact, we will show that every finite non-cyclic
group $G$ has such collections, and furthermore, that some of them
(which we denote \emph{unbalanced, }see Definition \ref{def:unbalanced_pair})
necessarily yield non-isometric quotients.

\subsection{\label{sub:ExampleZ2_Z2}Example}

Let $T$ be the torus $\nicefrac{\mathbb{R}^{2}}{\mathbb{Z}^{2}}$.
Let $G=\left\{ e,\sigma,\tau,\sigma\tau\right\} $ be the non-cyclic
group of size four (i.e., $G\cong\nicefrac{\mathbb{Z}}{2\mathbb{Z}}\times\nicefrac{\mathbb{Z}}{2\mathbb{Z}}$),
and let $\sigma,\tau\in G$ act on $T$ by two perpendicular rotations:
$\sigma\cdot\left(x,y\right)=\left(x,y+\frac{1}{2}\right)$ and $\tau\cdot\left(x,y\right)=\left(x+\frac{1}{2},y\right)$
(Figure \ref{fig:Z2Z2 action on the torus}).\FigBesBeg 
\begin{figure}[h]
\begin{minipage}[t]{0.6\columnwidth}%
\begin{center}
\includegraphics[scale=0.85]{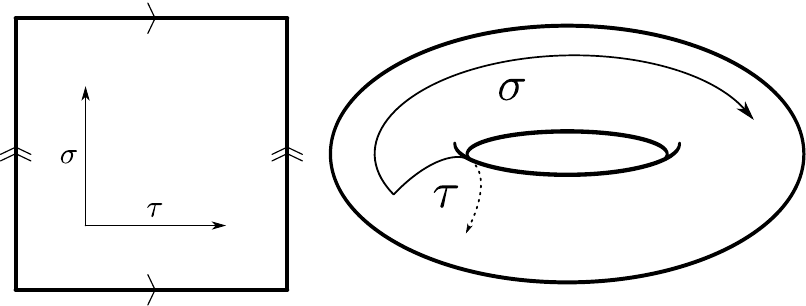}
\par\end{center}%
\end{minipage}\caption{\label{fig:Z2Z2 action on the torus}Two views of an action of $G=\left\{ e,\sigma,\tau,\sigma\tau\right\} \cong\nicefrac{\mathbb{Z}}{2\mathbb{Z}}\times\nicefrac{\mathbb{Z}}{2\mathbb{Z}}$
on the torus $T$.}
\end{figure}
\FigBesEnd 

Now, the subgroups
\begin{gather}
\begin{array}{l}
H_{1}=\left\{ e,\sigma\right\} \\
H_{2}=\left\{ e,\tau\right\} \\
H_{3}=\left\{ e,\sigma\tau\right\} 
\end{array}\qquad\begin{array}{l}
K_{1}=\left\{ e\right\} \\
K_{2}=K_{3}=G
\end{array}\label{eq:Hi_Ki_for_Z2xZ2}
\end{gather}
satisfy \eqref{eq:GenSunada-condition} (since $G$ is abelian, \eqref{eq:GenSunada-condition}
becomes 
\[
\forall g\in G:\quad\sum\limits _{i\,:\, g\in H_{i}}\frac{1}{\left|H_{i}\right|}=\sum\limits _{i\,:\, g\in K_{i}}\frac{1}{\left|K_{i}\right|}\,,
\]
which is easy to verify). Thus, the unions of tori $\bigcup\nicefrac{T}{H_{i}}=\nicefrac{T}{\left\langle \sigma\right\rangle }\bigcup\nicefrac{T}{\left\langle \tau\right\rangle }\bigcup\nicefrac{T}{\left\langle \sigma\tau\right\rangle }$
and $\bigcup\nicefrac{T}{K_{i}}=T\bigcup\nicefrac{T}{G}\bigcup\nicefrac{T}{G}$
are isospectral (Figure \ref{fig:Z2_torus_pair}).
\begin{figure}[H]
\centering{}\includegraphics[scale=0.5]{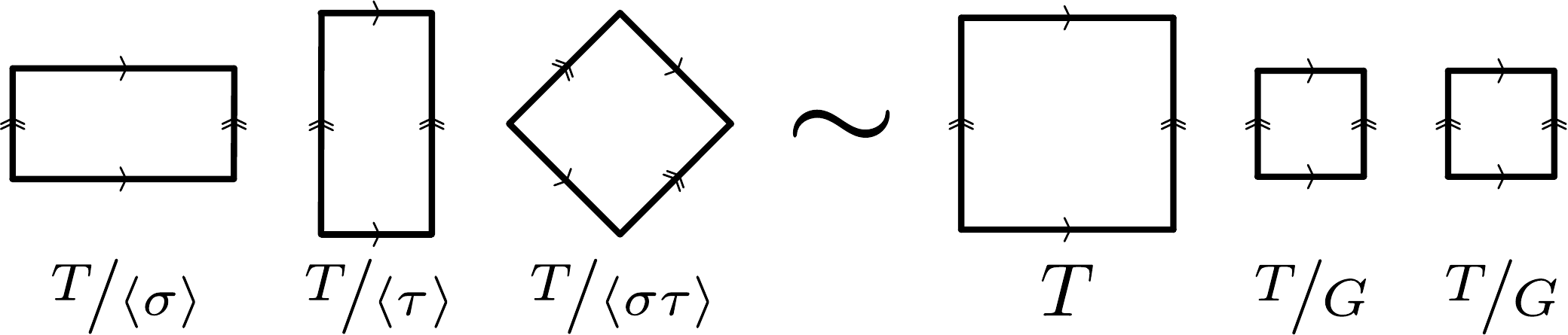}\caption{\label{fig:Z2_torus_pair}An isospectral pair consisting of quotients
of the torus $T$ (Figure \ref{fig:Z2Z2 action on the torus}) by
the subgroups of $G$ described in \eqref{eq:Hi_Ki_for_Z2xZ2}.}
\end{figure}
This isospectral pair, which we shall return to in section \ref{sub:Application---Hecke},
was immortalized in the words of Peter Doyle \cite{doyle2011laplace}:
\begin{quotation}
Two one-by-ones and a two-by-two,

Two two-by-ones and a roo-by-roo. 
\end{quotation}

\medskip{}

This paper is ogranized as follows. Section \ref{sec:G-sets} describes
the elements we shall need from the theory of $G$-sets: their classification,
linear equivalence, and tensor product. Section \ref{sec:Action-and-spectrum}
explains why tensoring a manifold with linearly equivalent $G$-sets
gives isospectral manifolds, and defines the notion of unbalanced
$G$-sets, which yield isospectral manifolds which are also non-isometric.
At this point the focus turns to the totality of isospectral pairs
arising from a single action, and it is shown that it posseses a natural
structure of a lattice. Section \ref{sec:Construction} is devoted
to the proof that every finite non-cyclic group admits an unbalanced
pair, and various isospectral pairs are encountered along the way.
Section \ref{sec:Computation} demonstrates a detailed computation
of (generators for) the lattice of isospectral pairs arising from
the symmetries of the regular hexagon. Finally, Section \ref{sec:Generalizations}
hints at possible generalizations of the results presented in this
paper.

\section{\label{sec:G-sets}$G$-sets}

To explain where the conditions \eqref{eq:Sunada-condition} and \eqref{eq:GenSunada-condition}
come from, we use the theory of $G$-sets. We start by recalling the
basic notions and facts.

\subsection{\label{sub:G-sets-classification}$G$-sets and their classification}

For a group $G$, a (left)\emph{ $G$-set} $X$ is a set equipped
with a (left) action of $G$, i.e., a multiplication rule $G\times X\rightarrow X$.
Such an action partitions $X$ into \emph{orbits}, the subsets of
the form $Gx=\left\{ gx\,\middle|\, g\in G\right\} $ for $x\in X$.
A $G$-set with one orbit is said to be \emph{transitive}, and every
$G$-set decomposes uniquely as a disjoint union of transitive ones,
its orbits. For every subgroup $H$ of $G$, the set of left cosets
$\nicefrac{G}{H}$ is a transitive (left) $G$-set.

We denote by $\Hom_{G}\left(X,Y\right)$ the set of $G$-set homomorphisms
from $X$ to $Y$, which are the functions $f:X\rightarrow Y$ which
commute with the actions, i.e., satisfy $f\left(gx\right)=gf\left(x\right)$
for all $g\in G$, $x\in X$. An isomorphism is, as usual, an invertible
homomorphism.

Every transitive $G$-set is isomorphic to $\nicefrac{G}{H}$, for
some subgroup $H$ of $G$, and $\nicefrac{G}{H}$ and $\nicefrac{G}{K}$
are isomorphic if and only if $H$ and $K$ are conjugate subgroups
of $G$. More generally, every $G$-set is isomorphic to $\bigcup_{i\in I}\nicefrac{G}{H_{i}}$
for some collection (possibly with repetitions) of subgroups $H_{i}$
$\left(i\in I\right)$ in $G$, and these are determined uniquely
up to order and conjugacy. I.e., $X=\bigcup\nicefrac{G}{H_{i}}$ and
$Y=\bigcup\nicefrac{G}{K_{i}}$ are isomorphic if and only if after
some reordering $H_{i}$ is conjugate to $K_{i}$ for every $i$.

A \emph{right $G$-set }is a set equipped with a right action of $G$,
i.e., a multiplication rule $X\times G\rightarrow X$ (satisfying
$x\left(gg'\right)=\left(xg\right)g'$). The classification of right
$G$-sets by right cosets is analogous to that of left $G$-sets by
left ones.

\subsection{Linearly equivalent $G$-sets}

Henceforth $G$ is a finite group, and all $G$-sets are finite, so
that every $G$-set is isomorphic to a finite disjoint union of the
form $\bigcup\nicefrac{G}{H_{i}}$. For a $G$-set $X$, $\mathbb{C}\left[X\right]$
denotes the complex representation of $G$ having $X$ as a basis,
with $G$ acting on $\mathbb{C}\left[X\right]$ by the linear extension
of its action on $X$, i.e., $g\sum a_{i}x_{i}=\sum a_{i}gx_{i}$
($g\in G$, $a_{i}\in\mathbb{C}$, $x_{i}\in X$).

If $X\cong Y$ (as $G$-sets), then $\mathbb{C}\left[X\right]\cong\mathbb{C}\left[Y\right]$
(as $\mathbb{C}G$-modules, i.e., complex representations), but not
vice versa. In fact, this is precisely where \eqref{eq:Sunada-condition}
and \eqref{eq:GenSunada-condition} come from:
\begin{prop}
\label{pro:equivalence_of_linequiv}For two (finite) $G$-sets $X,Y$
the following are equivalent:
\begin{enumerate}
\item $\mathbb{C}\left[X\right]\cong\mathbb{C}\left[Y\right]$ as complex
representations of $G$.
\item Every $g\in G$ fixes the same number of elements in $X$ and in $Y$.
\item $X\cong\bigcup\nicefrac{G}{H_{i}}$ and $Y\cong\bigcup\nicefrac{G}{K_{i}}$
for $H_{i},K_{i}\leq G$ satisfying \eqref{eq:GenSunada-condition}.
\end{enumerate}
\end{prop}
\begin{proof}
The character of $\mathbb{C}\left[X\right]$ is $\chi_{\mathbb{C}\left[X\right]}\left(g\right)=\left|\fix_{X}\left(g\right)\right|$,
hence by character theory $\mathit{\left(1\right)}$ is equivalent
to $\mathit{\left(2\right)}$. It is a simple exercise to show that
$\left|\fix_{\nicefrac{G}{H}}\left(g\right)\right|=\frac{\left|g^{G}\cap H\right|\left|\vphantom{g^{G}}C_{G}\left(g\right)\right|}{\left|H\right|}$,
hence for $H_{i}$ such that $X\cong\bigcup\nicefrac{G}{H_{i}}$ we
obtain 
\[
\left|\fix_{X}\left(g\right)\right|=\sum_{i}\left|\fix_{\nicefrac{G}{H_{i}}}\left(g\right)\right|=\left|\vphantom{g^{G}}C_{G}\left(g\right)\right|\cdot\sum_{i}\frac{\left|g^{G}\cap H_{i}\right|}{\left|H_{i}\right|},
\]
showing that $\mathit{\left(2\right)}$ is equivalent to $\mathit{\left(3\right)}$.\end{proof}
\begin{defn}
\label{def:lin_eq_sets}$G$-sets $X$ and $Y$ as in Proposition
\ref{pro:equivalence_of_linequiv} are said to be \emph{linearly equivalent.}\end{defn}
\begin{rem*}
In the literature one encounters also the terms \emph{arithmetically
equivalent}, \emph{almost equivalent}, \emph{Gassman pair},\emph{
}or \emph{Sunada pair}. Also, sometimes the {}``trivial case'',
i.e., when $X\cong Y$ as $G$-sets, is excluded.
\end{rem*}

\subsubsection{\label{sub:Z2_Z2_Gsets}Back to the example}

In \eqref{eq:Hi_Ki_for_Z2xZ2} we presented subgroups $H_{i},K_{i}$
of $G=\left\{ e,\sigma,\tau,\sigma\tau\right\} \cong\nicefrac{\mathbb{Z}}{2\mathbb{Z}}\times\nicefrac{\mathbb{Z}}{2\mathbb{Z}}$,
which satisfied condition \eqref{eq:GenSunada-condition}. Figure
\ref{fig:Z2xZ2_Gsets} shows the corresponding $G$-sets $X=\bigcup\nicefrac{G}{H_{i}}$
and $Y=\bigcup\nicefrac{G}{K_{i}}$, and one indeed sees that 
\[
\left|\fix_{X}\left(g\right)\right|=\left|\fix_{Y}\left(g\right)\right|=\begin{cases}
6 & g=e\\
2 & g=\sigma,\tau,\sigma\tau
\end{cases}
\]

\FigBesBeg 
\begin{figure}[h]
\caption{\label{fig:Z2xZ2_Gsets}$X$ and $Y$ are linearly equivalent $G$-sets
for $G=\left\{ e,\sigma,\tau,\sigma\tau\right\} $, corresponding
to the subgroups in \eqref{eq:Hi_Ki_for_Z2xZ2}.}

\begin{minipage}[t]{0.5\columnwidth}%
\begin{center}
\includegraphics{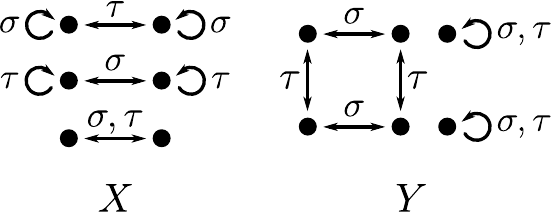}
\par\end{center}%
\end{minipage}
\end{figure}
We note that $X$ and $Y$ are not isomorphic as $G$-sets, as the
sizes of their orbits are different: $X$ has three orbits of size
two, whereas $Y$ has one orbit of size four and two orbits of size
one.

\subsubsection{\label{sub:The-transitive-case}The transitive case - Gassman-Sunada
pairs}

When restricting to transitive $G$-sets, $X$ and $Y$ are linearly
equivalent exactly when $X\cong\nicefrac{G}{H}$, $Y\cong\nicefrac{G}{K}$
for $H,K\leq G$ satisfying the Sunada condition \eqref{eq:Sunada-condition}.
In the literature $H,K$ are known as \emph{almost conjugate, locally
conjugate, arithmetically equivalent, linearly equivalent, Gassman
pair}, or\emph{ Sunada pair}, and again one usually excludes the trivial
case, which is when $H$ and $K$ are conjugate. For a group to have
a Sunada pair its order must be a product of at least five primes
\cite{dipasqualeorder}, but there exist such $n$ (the smallest being
80), for which no group of size $n$ has one. The smallest group which
has a Sunada pair is $\nicefrac{\mathbb{Z}}{8\mathbb{Z}}\rtimes\mathrm{Aut}\left(\nicefrac{\mathbb{Z}}{8\mathbb{Z}}\right)$
(of size 32).

\subsection{\label{sub:Tensor-product-of-Gsets}Tensor product of $G$-sets}

The theory of $G$-sets is parallel in many aspects to that of $R$-modules
(where $R$ stands for a non-commutative ring). This section describes
in some details the $G$-set analogue of the tensor product of modules.
Except for Definition \ref{def:The-tensor-product}, and the universal
property \eqref{eq:tensor_hom_adj}, this section may be skipped by
abstract nonsence haters.

If $M$ is a right $R$-module, for every abelian group $A$ the group
of homomorphisms $\Hom_{Ab}\left(M,A\right)$ has a structure of a
(left) $R$-module, by $\left(rf\right)\left(m\right)=f\left(mr\right)$.
In fact, $\Hom_{Ab}\left(M,\_\right)$ is a functor from $Ab$ to
$Rmod$, the category of left $R$-modules. This functor has a celebrated
left adjoint, the tensor product $M\otimes_{R}\_:Rmod\rightarrow Ab$.
This means that for every $R$-module $N$ there is an isomorphism
\[
\Hom_{Ab}\left(M\otimes_{R}N,A\right)\cong\Hom_{R}\left(N,\Hom_{Ab}\left(M,A\right)\right)
\]
which is natural in $N$ and $A$. 

The analogue for $G$-sets is this: if $X$ is a right $G$-set,
then for every set $S$ the set $\Hom_{Set}\left(X,S\right)$ has
a structure of a (left) $G$-set, by $\left(gf\right)\left(x\right)=f\left(xg\right)$.
Here $\Hom_{Set}\left(X,\_\right)$ is a functor from $Set$ to $Gset$
(the category of left $G$-sets), and again it has a left adjoint:
\begin{defn}
\label{def:The-tensor-product}The tensor product over $G$ of a right
$G$-set $X$ and a left $G$-set $Y$, denoted $X\times_{G}Y$, is
the set $\nicefrac{X\times Y}{\left(xg,y\right)\sim\left(x,gy\right)}$,
i.e., the quotient set of the cartesian product $X\times Y$ by the
relations $\left(xg,y\right)\sim\left(x,gy\right)$ (for all $x\in X$,
$g\in G$, $y\in Y$).
\end{defn}
The functor $X\times_{G}\_:Gset\rightarrow Set$ is indeed the left
adjoint of $\Hom_{Set}\left(X,\_\right)$, i.e., for every $G$-set
$Y$ there is an isomorphism (natural in $Y$ and $S$)
\[
\Hom_{Set}\left(X\times_{G}Y,S\right)\cong\Hom_{G}\left(Y,\Hom_{Set}\left(X,S\right)\right).
\]
As it is custom to write $B^{A}$ for $\Hom_{Set}\left(A,B\right)$,
this can be written as
\begin{equation}
S^{X\times_{G}Y}\cong\Hom_{G}\left(Y,S^{X}\right)\label{eq:tensor_hom_adj}
\end{equation}
which for $G=1$ is the familiar isomorphism of sets $S^{X\times Y}\cong\left(S^{X}\right)^{Y}$.

The tensor product of $G$-sets behaves much like that of modules,
e.g., there are natural isomorphisms as follows:
\begin{itemize}
\item Distributivity: $\left(\bigcup X_{i}\right)\times_{G}Y\cong\bigcup\left(X_{i}\times_{G}Y\right)$.
\item Associativity: $\left(X\times_{G}Y\right)\times_{H}Z\cong X\times_{G}\left(Y\times_{H}Z\right)$
(where $Y$ is a $\left(G,H\right)$-biset, i.e., $\left(gy\right)h=g\left(yh\right)$
holds for all $g\in G$, $y\in Y$, $h\in H$).
\item Neutral element: $G\times_{G}X\cong X$.
\item Extension of scalars: if $H\leq G$, $G$ is a $\left(G,H\right)$-biset.
For an $H$-set $X$, this gives $G\times_{H}X$ a $G$-set structure
(by $g'\left(g,x\right)=\left(g'g,x\right)$). This construction
is adjoint to the restriction of scalars, i.e., for a $G$-set $Y$
one has 
\begin{equation}
\Hom_{G}\left(G\times_{H}X,Y\right)\cong\Hom_{H}\left(X,Y\right).\label{eq:extension_adjointness}
\end{equation}
\end{itemize}
\begin{rem*}
A point in which groups and rings differ is the following: a left
$G$-set can be regarded as a right one, by defining the right action
to be $xg=g^{-1}x$. Thus, we shall allow ourselves to regard left
$G$-sets as a right ones, and vice versa%
\footnote{For rings, a left $R$-module can only be regarded as a right $R^{\mathrm{opp}}$-module,
and in general $R\ncong R^{opp}$. In groups, $G\cong G^{\mathrm{opp}}$
canonically by the inverse map.%
}. Going back to Definition \ref{def:The-tensor-product}, if we choose
to regard $X$ as a left $G$-set, we get
\[
X\times_{G}Y=\frac{X\times Y}{\left(xg,y\right)\sim\left(x,gy\right)}=\frac{X\times Y}{\left(g^{-1}x,y\right)\sim\left(x,gy\right)}=\frac{X\times Y}{\left(x,y\right)\sim\left(gx,gy\right)}=\nicefrac{X\times Y}{G}
\]
i.e., the tensor product is the orbit set of the normal (cartesian)
product of the left $G$-sets $X$ and $Y$. A word of caution: the
process of turning a left $G$-set into a right one does not give
it, in general, a $\left(G,G\right)$-biset structure.
\end{rem*}

\section{\label{sec:Action-and-spectrum}Action and spectrum}

\subsection{Tensor product of $G$-manifolds}

Assume we have an action of $G$ on a Riemannian manifold $M$ and
on a finite $G$-set $X$. Our purpose is to study $M\times_{G}X$,
which has a Riemannian orbifold structure as a quotient of $M\times X$
(where $X$ is given the discrete topology)%
\footnote{More generally, if $M$ and $M'$ are $G$-manifolds, $M\times_{G}M'$
is an orbifold (manifold, if $G$ acts freely on $M\times M'$), but
in this paper we shall only consider the tensor product of a $G$-manifold
and a finite $G$-set (which can be regarded as a compact manifold
of dimension $0$).%
}. In Section \ref{sec:Intro} we discussed unions of the form $\bigcup\nicefrac{M}{H_{i}}$
for subgroups $H_{i}\leq G$, and this is still our object of study:
we can choose subgroups $H_{i}$ of $G$ such that $X\cong\bigcup\nicefrac{G}{H_{i}}$,
and for any such choice we have an isometry $M\times_{G}X\cong\bigcup\nicefrac{M}{H_{i}}$.
This can be verified directly, or by the tensor properties:
\begin{align*}
M\times_{G}X & \cong M\times_{G}\left(\bigcup\nicefrac{G}{H_{i}}\right)\cong\bigcup\left(M\times_{G}\nicefrac{G}{H_{i}}\right)\cong\bigcup\left(M\times_{G}\left(G\times_{H_{i}}\mathbf{1}\right)\right)\\
 & \cong\bigcup\left(\left(M\times_{G}G\right)\times_{H_{i}}\mathbf{1}\right)\cong\bigcup\left(M\times_{H_{i}}\mathbf{1}\right)\cong\bigcup\nicefrac{M}{H_{i}}
\end{align*}
where $\mathbf{1}$ denotes a one-element set. In this light, the
tensor product generalizes the notion of quotients, since quotients
by subgroups of $G$ correspond to tensoring with transitive $G$-sets:
$\nicefrac{M}{H}\cong M\times_{G}\nicefrac{G}{H}$. The advantage
of studying $M\times_{G}X$ instead of $\bigcup\nicefrac{M}{H_{i}}$
is that the former is free of choices, and thus more suitable for
functorial constructions, and yields more elegant proofs. On the other
hand, $\bigcup\nicefrac{M}{H_{i}}$ is much more familiar, and the
reader is encouraged to envision $M\times_{G}X$ as a union of quotients
of $M$.
\begin{thm}
\label{thm:product_spectrum}If $G$ acts on a Riemannian manifold
$M$ then for every finite $G$-set $X$ there is an isomorphism
\[
L^{2}\left(M\times_{G}X\right)\cong\Hom_{\mathbb{C}G}\left(\mathbb{C}\left[X\right],L^{2}\left(M\right)\right).
\]
(here $L^{2}\left(M\right)$ is a representation of $G$ by $\left(gf\right)\left(m\right)=f\left(g^{-1}m\right)$,
for $g\in G$, $f\in L^{2}\left(M\right)$, $m\in M$).\end{thm}
\begin{rem*}
In the language of \cite{band2009isospectral,parzanchevski2010linear},
this means that $M\times_{G}X$ is an $\nicefrac{M}{\mathbb{C}\left[X\right]}$-manifold,
and since $M\times_{G}X\cong\bigcup\nicefrac{M}{H_{i}}$, this is
implied in Section 9.3 of \cite{band2009isospectral}. However, the
perspective of tensor product gives a direct proof.\end{rem*}
\begin{proof}
We have isomorphisms of vector spaces 
\begin{equation}
\mathbb{C}^{M\times_{G}X}\cong\Hom_{G}\left(X,\mathbb{C}^{M}\right)\cong\Hom_{\mathbb{C}G}\left(\mathbb{C}\left[X\right],\mathbb{C}^{M}\right).\label{eq:Hom_MGX}
\end{equation}
The left isomorphism is by adjointness of tensor and hom \eqref{eq:tensor_hom_adj},
and it is given explicitly by sending $f\in\mathbb{C}^{M\times_{G}X}$
to $F\in\Hom_{G}\left(X,\mathbb{C}^{M}\right)$ defined by $F\left(x\right)\left(m\right)=f\left(m,x\right)$.
The next isomorphism is by adjointness of the free construction $X\mapsto\mathbb{C}\left[X\right]$
and the forgetful functor $\mathbb{C}Gmod\rightarrow Gset$, i.e.,
\begin{equation}
\Hom_{G}\left(X,\_\right)\cong\Hom_{\mathbb{C}G}\left(\mathbb{C}\left[X\right],\_\right),\label{eq:free-CG-module}
\end{equation}
and is given explicitly by linear extension, i.e., defining $F\left(\sum a_{i}x_{i}\right)=\sum a_{i}F\left(x_{i}\right)$.
The correspondence of the $L^{2}$ conditions then follows from the
finiteness of $G$ and $X$, and the fact that $\int_{M\times X}\left|f\right|^{2}=\sum_{x\in X}\int_{M}\left|f\left(\,\cdot\,,x\right)\right|^{2}$.\end{proof}
\begin{defn}
\label{def:spectrum}The spectrum of a Riemannian manifold $M$ is
the function $\mathrm{Spec}_{M}:\mathbb{R}\rightarrow\mathbb{N}$
which perscribes to every number its multiplicity as an eigenvalue
of the Laplace operator on $M$, i.e., $\Spec_{M}\left(\lambda\right)=\dim L_{\lambda}^{2}\left(M\right)$
where $L_{\lambda}^{2}\left(M\right)=\left\{ f\in L^{2}\left(M\right)\,\middle|\,\Delta f=\lambda f\right\} $.\end{defn}
\begin{cor}
\label{cor:isospectrality}If $G$ acts on $M$, and $X$ and $Y$
are linearly equivalent $G$-sets, then $M\times_{G}X$ and $M\times_{G}Y$
are isospectral.\end{cor}
\begin{rem*}
For transitive $X$ and $Y$, this is equivalent to Sunada's theorem.\end{rem*}
\begin{proof}
By Theorem \ref{thm:product_spectrum}, we have $L^{2}\left(M\times_{G}X\right)\cong L^{2}\left(M\times_{G}Y\right)$,
but we must verify that this isomorphism respects the Laplace operator.
If $y\mapsto\sum_{x\in X}a_{y,x}x$ is a $\mathbb{C}G$-module isomorphism
from $\mathbb{C}\left[Y\right]$ to $\mathbb{C}\left[X\right]$, then
$\mathcal{T}:L^{2}\left(M\times_{G}X\right)\overset{\cong}{\longrightarrow}L^{2}\left(M\times_{G}Y\right)$
is given explicitly by $\left(\mathcal{T}f\right)\left(m,y\right)=\sum_{x\in X}a_{y,x}f\left(m,x\right)$
($\mathcal{T}$ is a \emph{transplantation} map, see \cite{buser1986isospectral,conway1994some,chapman1995drums}).
This isomorphism commutes with the Laplace operators on their domains
of definition, hence inducing isomorphism of eigenspaces, and in particular
equality of spectra. Alternatively, one can replace $L^{2}$ throughout
Theorem \ref{thm:product_spectrum} with $L_{\lambda}^{2}$, obtaining
directly $L_{\lambda}^{2}\left(M\times_{G}X\right)\cong\Hom_{\mathbb{C}G}\left(\mathbb{C}\left[X\right],L_{\lambda}^{2}\left(M\right)\right)$,
and thus $L_{\lambda}^{2}\left(M\times_{G}X\right)\cong L_{\lambda}^{2}\left(M\times_{G}Y\right)$.
\end{proof}
The theorem and corollary above give us isospectral manifolds, but
do not tell us whether they are isometric or not. First of all, if
$X$ and $Y$ are isomorphic as $G$-sets then $M\times_{G}X$ and
$M\times_{G}Y$ are certainly isometric. However, this may happen
also for non-isomorphic $G$-sets%
\footnote{For example, if $H$ and $K$ are isomorphic subgroups of $G$, and
the action of $G$ on $M$ can be extended to an action of some supergroup
$\widehat{G}$ in which $H$ and $K$ are conjugate, then $\nicefrac{M}{H}$
and $\nicefrac{M}{K}$ are also isometric.%
}. The next section deals with this inconvenience.

\subsection{Unbalanced pairs }

In Section \ref{sub:Z2_Z2_Gsets} we concluded that the $G$-sets
$X$ and $Y$ in Figure \ref{fig:Z2xZ2_Gsets} were non-isomorphic
by pointing out differences in the sizes of their orbits. This property
is stronger than just being non-isomorphic, and we give it a name.
\begin{defn}
\label{def:unbalanced_pair}For a finite group $G$, a pair of finite
$G$-sets $X,Y$ is an \emph{unbalanced pair }if they are linearly
equivalent (i.e., $\mathbb{C}\left[X\right]\cong\mathbb{C}\left[Y\right]$
as $\mathbb{C}G$-modules), and if in addition they differ in the
sizes of their orbits, i.e., for some $n$ the number of orbits of
size $n$ in $X$ and the number of such orbits in $Y$ are different.\end{defn}
\begin{rem}
\label{rem:No-unbalanced-Sunada}Since the size of a $G$-set $X$
equals $\dim\mathbb{C}\left[X\right]$, and the number of orbits in
$X$ equals $\dim\left(\mathbb{C}\left[X\right]^{G}\right)$ %
\footnote{$V^{G}$ denotes the $G$-invariant part of a representation $V$:
$V^{G}=\left\{ v\in V\,\middle|\, gv=v\:\forall g\in G\right\} $.%
}, linearly equivalent $G$-sets necessarily have the same size and
number of orbits. Thus, there are no unbalanced pairs in which one
of the sets is transitive, and in particular there are no unbalanced
Sunada pairs.\end{rem}
\begin{thm}
\label{thm:isospectral_nonisom}If $X,Y$ is an unbalanced pair of
$G$-sets, then for any faithful action of $G$ on a compact connected
manifold $M$ the manifolds (or orbifolds) $M\times_{G}X$ and $M\times_{G}Y$
are isospectral and non-isometric.\end{thm}
\begin{proof}
Isospectrality was obtained in Corollary \ref{cor:isospectrality}.
To show that $M\times_{G}X$ and $M\times_{G}Y$ cannot be isometric,
we choose $H_{i}$ such that $X\cong\bigcup\nicefrac{G}{H_{i}}$,
and observe that
\begin{itemize}
\item Since $M$ is connected, $\left\{ \nicefrac{M}{H_{i}}\right\} $ form
the connected components of $M\times_{G}X$.
\item Since $G$ acts faithfully and $M$ is connected, $\vol\nicefrac{M}{H_{i}}=\frac{\vol M}{\left|H_{i}\right|}$.
\end{itemize}
Thus, the sizes of orbits in $X$ correspond to the volumes of connected
components in $M\times_{G}X$ %
\footnote{This correspondence between sizes of orbits and volumes of components
is apparent in Figures \ref{fig:Z2xZ2_Gsets} and \ref{fig:Z2_torus_pair}.%
}. Therefore, if $X$ and $Y$ form an unbalanced pair then $M\times_{G}X$
and $M\times_{G}Y$ differ in the volumes of their connected components.
To be precise, if $X$ and $Y$ have different numbers of orbits of
size $n$, then $M\times_{G}X$ and $M\times_{G}Y$ have different
numbers of connected components of volume $\frac{n\cdot\vol M}{\left|G\right|}$.
\end{proof}

\subsection{\label{sub:Rings}The Burnside ring and the lattice of isospectral
quotients}

A nice point of view is attained from $\Omega\left(G\right)$, the
Burnside ring of the group $G$. Its elements are formal differences
of isomorphism classes of finite G-sets, i.e., $X-Y$, where $X$
and $Y$ are finite $G$-sets, with $X-Y=X'-Y'$ whenever $X\cup Y'\cong X'\cup Y$.
The operations in $\Omega\left(G\right)$ are disjoint union and cartesian
product (extended to formal differences by distributivity). If we
fix representatives $H_{1},\ldots,H_{r}$ for the conjugacy classes
of subgroups in $G$, the classification of $G$-sets (Section \ref{sub:G-sets-classification})
tells us that $\Omega\left(G\right)=\left\{ \sum_{i=1}^{r}n_{i}\cdot\nicefrac{G}{H_{i}}\,\middle|\, n_{i}\in\mathbb{Z}\right\} $,
so that as an abelian group $\Omega\left(G\right)^{+}\cong\mathbb{Z}^{r}$
with $\left\{ \nicefrac{G}{H_{i}}\right\} _{i=1}^{r}$ being a basis.

Now, instead of looking at a pair of $G$-sets $\left(X,Y\right)$,
we look at the element $X-Y$ in $\Omega\left(G\right)$. First, we
note that some information is lost: for any $G$-set $Z$, the pair
$\left(X,Y\right)$ and the pair $\left(X'=X\cup Z,Y'=Y\cup Z\right)$
both correspond to the same element in $\Omega\left(G\right)$, i.e.,
$X-Y=X'-Y'$. Second, we notice this is in fact desirable. In order
to produce elegant isospectral pairs, one would like to {}``cancel
out'' isometric connected components shared by two isospectral manifolds
(as in \cite{chapman1995drums}), and the pair $M\times_{G}X'$, $M\times_{G}Y'$
is just the pair $M\times_{G}X$, $M\times_{G}Y$ with each manifold
added $M\times_{G}Z$.

Thus, we would like to look at \emph{reduced pairs}, pairs of $G$-sets
$X,Y$ which share no isomorphic sub-$G$-sets (equivalently, no isomorphic
orbits). The map $\left(X,Y\right)\mapsto X-Y$ gives a correspondence
between reduced pairs and the elements of $\Omega\left(G\right)$%
\footnote{Just like the map $\left(x,y\right)\mapsto\frac{x}{y}$ gives a correspondence
between reduced pairs of positive integers ($x,y\in\mathbb{N}$ such
that $\gcd\left(x,y\right)=1$), and positive rationals.%
}. Since $X\cong Y$ if and only if $X-Y=0$, nonzero elements in $\Omega\left(G\right)$
correspond to reduced pairs of non-isomorphic $G$-sets, and $0$
corresponds to the (reduced) pair $\left(\varnothing,\varnothing\right)$.

A second ring of interest is $R\left(G\right)$, the representation
ring of $G$. Its elements are formal differences of isomorphism classes
of complex representations of $G$, with the operations being direct
sum and tensor product. $R\left(G\right)$ also denotes the ring of
virtual characters of $G$, which is isomorphic to the representation
ring (see, e.g., \cite{serre1977linear})%
\footnote{As an abelian group $R\left(G\right)$ can also be identified with
$K_{0}\mathbb{C}G$.%
}. There is a ring homomorphism from $\Omega\left(G\right)$ into $R\left(G\right)$,
given by $X\mapsto\mathbb{C}\left[X\right]$ (or $X\mapsto\chi_{\mathbb{C}\left[X\right]}$,
considering $R\left(G\right)$ as the character ring). We denote the
kernel of this homomorphism by $\mathcal{L}\left(G\right)$, and say
that its elements are \emph{linearly trivial}. The formal difference
$X-Y$ is in $\mathcal{L}\left(G\right)$ if and only if $\mathbb{C}\left[X\right]\cong\mathbb{C}\left[Y\right]$,
so that we have a correspondence between linearly trivial elements
in $\Omega\left(G\right)$ and reduced pairs of linearly equivalent
$G$-sets. 

Since $\mathcal{L}\left(G\right)$, the ideal of linearly trivial
elements, is a subgroup of the free abelian group $\Omega\left(G\right)^{+}\cong\mathbb{Z}^{r}$,
it is also free abelian: $\mathcal{L}\left(G\right)\cong\mathbb{Z}^{m}$
for some $m\leq r$. This means that we can find a $\mathbb{Z}$-basis
for $\mathcal{L}\left(G\right)$ (we demonstrate how to compute such
a basis in Section \ref{sec:Computation}). This gives a lattice of
linearly equivalent reduced pairs, as follows: if $\left\{ X_{i}-Y_{i}\right\} _{i=1..m}$
is a basis for $\mathcal{L}\left(G\right)$, and we define for $\bar{n}=\left(n_{1},\ldots,n_{m}\right)\in\mathbb{Z}^{m}$
\begin{align*}
X_{\bar{n}} & =\left(\bigcup_{\, i\,:\, n_{i}>0}n_{i}X_{i}\right)\cup\left(\bigcup_{\, i\,:\, n_{i}<0}\left|n_{i}\right|Y_{i}\right)\\
Y_{\bar{n}} & =\left(\bigcup_{\, i\,:\, n_{i}<0}\left|n_{i}\right|X_{i}\right)\cup\left(\bigcup_{\, i\,:\, n_{i}>0}n_{i}Y_{i}\right)
\end{align*}
then every reduced pair of linearly equivalent $G$-sets $\left(X,Y\right)$
is obtained by canceling out common factors in $\left(X_{\bar{n}},Y_{\bar{n}}\right)$,
for a unique $\bar{n}\in\mathbb{Z}^{m}$. 

Given an action of $G$ on a manifold $M$, we associate with every
$G$-set $X$ a manifold, namely $M\times_{G}X$. The lattice of linearly
equivalent pairs then maps to a lattice of isospectral pairs (see
the example in Section \ref{sec:Computation}). For a general manifold
$M$, this might be only a sublattice of the lattice of isospectral
quotients, which can be described as follows. We pull the spectrum
function backwards to $\Omega\left(G\right)$, defining $\Spec_{X-Y}=\Spec_{M\times_{G}X}-\Spec_{M\times_{G}Y}$
(so that we have $\Spec:\Omega\left(G\right)\rightarrow\mathbb{Z}^{\mathbb{R}}$).
Isospectral pairs of the form $\left(M\times_{G}X,M\times_{G}Y\right)$
are exactly those for which $X-Y\in\ker\Spec$, and Corollary \ref{cor:isospectrality}
states that this kernel (for any $M$) contains $\mathcal{L}\left(G\right)$.

\section{\label{sec:Construction}Construction of unbalanced pairs}

Our objective in this section is to find unbalanced pairs. That is,
given a group $G$, to find two $G$-sets $X,Y$ which differ in the
number of orbits of some size, and such that $\mathbb{C}\left[X\right]\cong\mathbb{C}\left[Y\right]$
as $\mathbb{C}G$-modules. We shall do so by {}``balancing'' unions
of transitive $G$-sets, which correspond to coset spaces of the form
$\nicefrac{G}{H}$. For every subgroup $H\leq G$ we denote by $\mathcal{S}_{H}$
the function
\begin{equation}
\mathcal{S}_{H}\left(g\right)=\chi_{\mathbb{C}\left[\nicefrac{G}{H}\right]}\left(g\right)=\left|\fix_{\nicefrac{G}{H}}\left(g\right)\right|=\frac{\left|g^{G}\cap H\right|\left|\vphantom{g^{G}}C_{G}\left(g\right)\right|}{\left|H\right|}\label{eq:quasiregular_def}
\end{equation}
$\mathbb{C}\left[\nicefrac{G}{H}\right]$ is sometimes called the
\emph{quasiregular representation} of $G$ on $H$, and $\mathcal{S}_{H}$
is thus the \emph{quasiregular character}. It also bears the names
$\mathbf{1}_{H}^{G}$, $\mathbf{1}\negthickspace\uparrow_{H}^{G}$,
or $\mathrm{Ind}_{H}^{G}\mathbf{1}$, being the induction of the trivial
character of $H$ to $G$. Lastly, it is the image of $\nicefrac{G}{H}$
under the map $\Omega\left(G\right)\rightarrow R\left(G\right)$,
when the latter is regarded as the ring of virtual characters of $G$.

In light of Proposition \ref{pro:equivalence_of_linequiv}, we shall
seek $H_{i}$, $K_{i}$ such that $\sum_{i}\mathcal{S}_{H_{i}}=\sum_{i}\mathcal{S}_{K_{i}}$,
and then check that the obtained linearly equivalent pair is unbalanced.
We use a few easy calculations:
\begin{enumerate}
\item For the trivial subgroup $1\leq G$, we have 
\begin{equation}
\mathcal{S}_{1}\left(g\right)=\begin{cases}
\left|G\right| & g=e\\
0 & g\neq e
\end{cases}\label{eq:regular_character}
\end{equation}

\item For $H=G$ we have 
\begin{equation}
\mathcal{S}_{G}\equiv1\label{eq:trivial_character}
\end{equation}

\item For any $H$ we have 
\begin{equation}
\mathcal{S}_{H}\left(e\right)=\left[G:H\right]\label{eq:character_index}
\end{equation}

\item For $G$ abelian $g^{G}=\left\{ g\right\} $ and $C_{G}\left(g\right)=G$,
so that $\mathcal{S}_{H}=\left[G:H\right]\cdot\mathbf{1}_{H}$, i.e.,
\begin{equation}
\mathcal{S}_{H}\left(g\right)=\begin{cases}
\left[G:H\right] & g\in H\\
0 & g\notin H
\end{cases}\label{eq:Abelian_nu}
\end{equation}

\end{enumerate}

\subsection{\label{sub:Cyclic-groups}Cyclic groups}

Finite cyclic groups have no unbalanced pairs. This follows from the
following:
\begin{prop}
\label{pro:Finite-cyclic-groups}If $G$ is finite cyclic, linearly
equivalent $G$-sets are isomorphic.\end{prop}
\begin{proof}
Let $G=\nicefrac{\mathbb{Z}}{n\mathbb{Z}}$, and $D=\left\{ d\,\middle|\, d>0,\, d\mid n\right\} $.
The subgroups of $G$ are $H_{d}=\left\langle d\right\rangle $ for
$d\in D$, and by \eqref{eq:Abelian_nu} $\mathcal{S}_{H_{d}}=\frac{n}{d}\cdot\boldsymbol{1}_{H_{d}}$.
A non-trivial pair of linearly equivalent $G$-sets corresponds to
two different $\mathbb{N}$-combinations of $\left\{ \mathcal{S}_{H_{d}}\right\} _{d\in D}$
that agree as functions. Finding such a pair is equivalent to finding
a nonzero $\mathbb{Z}$-combination of $\left\{ \mathcal{S}_{H_{d}}\right\} _{d\in D}$
which vanishes. However, the matrix $\left(\mathcal{S}_{H_{d}}\left(d'\right)\right)_{d,d'\in D}$
is upper triangular with non-vanishing diagonal, which means that
$\left\{ \mathcal{S}_{H_{d}}\Big|_{D}\right\} _{d\in D}$ are linearly
independent over $\mathbb{Q}$, hence so are $\left\{ \mathcal{S}_{H_{d}}\right\} _{d\in D}$.
\end{proof}

\subsection{\label{sub:Case-I}$G=\nicefrac{\mathbb{Z}}{p\mathbb{Z}}\times\nicefrac{\mathbb{Z}}{p\mathbb{Z}}$}

Here we generalize the pair which appeared in Sections \ref{sub:ExampleZ2_Z2}
and \ref{sub:Z2_Z2_Gsets}. Let $p$ be a prime. $G=\nicefrac{\mathbb{Z}}{p\mathbb{Z}}\times\nicefrac{\mathbb{Z}}{p\mathbb{Z}}$
has $p+1$ subgroups of size (and index) $p$: $H_{\lambda}=\left\{ \left(x,y\right)\,\middle|\,\frac{x}{y}=\lambda\right\} $,
where $\lambda\in P^{1}\left(\mathbb{F}_{p}\right)=\left\{ 0,1,..,p-1,\infty\right\} $.
Every non-identity element in $G$ appears in exactly one of these,
and we obtain by \eqref{eq:character_index} and \eqref{eq:Abelian_nu}
\[
\sum_{\lambda\in P^{1}\left(\mathbb{F}_{p}\right)}\mathcal{S}_{H_{\lambda}}\left(g\right)=\begin{cases}
p\left(p+1\right) & g=e\\
p & g\neq e
\end{cases}\:.
\]
Consulting \eqref{eq:regular_character} and \eqref{eq:trivial_character},
we find that this is the same as $p\cdot\mathcal{S}_{G}+\mathcal{S}_{1}$,
so there is linear equivalence between 
\[
X=\bigcup\limits _{\lambda\in P^{1}\left(\mathbb{F}_{p}\right)}\nicefrac{G}{H_{\lambda}}\qquad\mathrm{and}\qquad Y=\underbrace{\mathbf{1}\cup\ldots\cup\mathbf{1}}_{p}\cup\, G\:,
\]
where $\mathbf{1}$ denotes the $G$-set with one element (corresponding
to $\nicefrac{G}{G}$). Obviously, this is an unbalanced pair ($X$
has $p+1$ orbits of size $p$, and $Y$ has one orbit of size $p^{2}$
and $p$ orbits with a single element). Figure \ref{fig:Z2xZ2_Gsets}
shows $X,Y$ for $p=2$ (by their Schreier graphs with respect to
the standard basis of $\nicefrac{\mathbb{Z}}{2\mathbb{Z}}\times\nicefrac{\mathbb{Z}}{2\mathbb{Z}}$).

\subsubsection{\label{sub:Application---Hecke}Application - Hecke pairs}

We now let 
\[
G=\left\langle \sigma,\tau\,\middle|\,\sigma^{p}=\tau^{p}=1,\sigma\tau=\tau\sigma\right\rangle \cong\nicefrac{\mathbb{Z}}{p\mathbb{Z}}\times\nicefrac{\mathbb{Z}}{p\mathbb{Z}}
\]
act on the torus $T=\nicefrac{\mathbb{R}^{2}}{\mathbb{Z}^{2}}$ by
the rotations $\sigma\cdot\left(x,y\right)=\left(x,y+\frac{1}{p}\right)$
and $\tau\cdot\left(x,y\right)=\left(x+\frac{1}{p},y\right)$. From
the unbalanced pair $X,Y$ constructed for $\nicefrac{\mathbb{Z}}{p\mathbb{Z}}\times\nicefrac{\mathbb{Z}}{p\mathbb{Z}}$
above one obtains the isospectral pair $T\times_{G}X$ and $T\times_{G}Y$,
each a union of $p+1$ tori. These examples were constructed using
different techniques by Doyle and Rossetti, who baptized them {}``Hecke
pairs'' \cite{doyle2011laplace}. The cases $p=2,3,5$ are illustrated
in Figure \ref{fig:2-3-5-tori}. One can verify that the analogue
pair for $p=4$, for example, is not isospectral - the reason is that
unlike in the prime case the subgroups 
\[
H_{\lambda}=\begin{cases}
\left\{ \left(x,\lambda x\right)\,\middle|\, x\in\nicefrac{\mathbb{Z}}{4\mathbb{Z}}\right\}  & \lambda=0..3\\
\left\{ \left(0,x\right)\,\middle|\, x\in\nicefrac{\mathbb{Z}}{4\mathbb{Z}}\right\}  & \lambda=\infty
\end{cases}
\]
do not cover $\left(\nicefrac{\mathbb{Z}}{4\mathbb{Z}}\times\nicefrac{\mathbb{Z}}{4\mathbb{Z}}\right)\backslash\left\{ 0\right\} $
evenly.

\FigBesEnd 

\begin{figure}[h]
\centering{}\includegraphics[width=1\columnwidth]{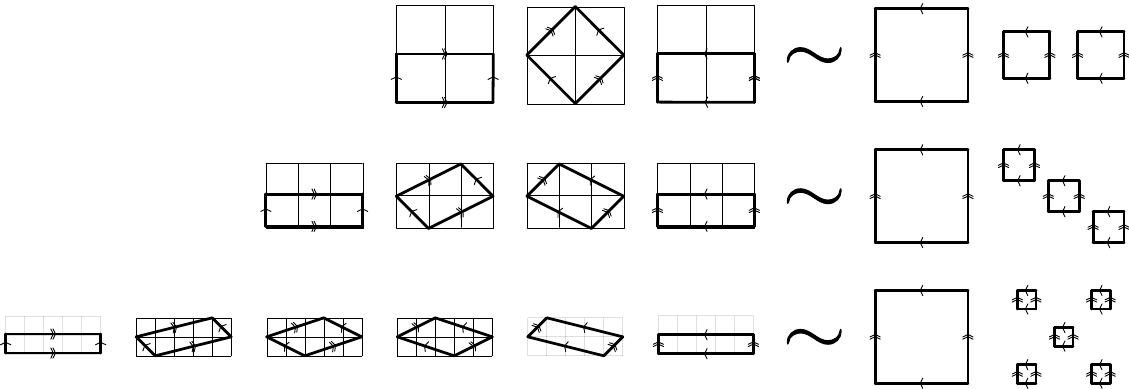}\caption{\label{fig:2-3-5-tori}Isospectral pairs consisting of unions of tori,
obtained from the action of $\nicefrac{\mathbb{Z}}{p\mathbb{Z}}\times\nicefrac{\mathbb{Z}}{p\mathbb{Z}}$
on the torus $\nicefrac{\mathbb{R}^{2}}{\mathbb{Z}^{2}}$, for $p=2,3,5$.
Grids are drawn to clarify the sizes.}
\end{figure}

\begin{rem*}
Since the spectrum of a flat torus is represented by a quadratic form,
isospectrality between flat tori can be interpreted as equality in
the representation numbers of forms%
\footnote{This insight (in the opposite direction) led Milnor to the first construction
of isospectral manifolds \cite{milnor1964eigenvalues}.%
}. For example, isospectrality in the case $p=2$ (Figure \ref{fig:2-3-5-tori},
top) asserts that together the quadratic forms $4m^{2}+n^{2}$, $2m^{2}+2n^{2}$
and $4m^{2}+n^{2}$ represent (over the integers) every value the
same number of times as do $m^{2}+n^{2}$, $4m^{2}+4n^{2}$, and $4m^{2}+4n^{2}$
together.
\end{rem*}

\subsection{\label{sub:Case-II}$G=\nicefrac{\mathbb{Z}}{q\mathbb{Z}}\rtimes\nicefrac{\mathbb{Z}}{p\mathbb{Z}}$}

Now let $G$ be the non-abelian group of size $pq$, where $p$ and
$q$ are primes such that $q\equiv1\left(\mathrm{mod}\, p\right)$.
$G$ has one subgroup $Q$ of size $q$, and $q$ subgroups $P_{1},P_{2},\ldots,P_{q}$
of size $p$. Since $Q$ is normal we have 
\[
g^{G}\cap Q=\begin{cases}
g^{G} & g\in Q\\
\varnothing & g\notin Q
\end{cases}\qquad\Rightarrow\qquad\mathcal{S}_{Q}\left(g\right)=\begin{cases}
p & g\in Q\\
0 & g\notin Q
\end{cases}\:.
\]
Every non-identity element of $G$ generates its entire centralizer,
for otherwise it would be in the center. Thus for $g\neq e$ 
\[
\sum_{i=1}^{q}\mathcal{S}_{P_{i}}\left(g\right)=\frac{\left|C_{G}\left(g\right)\right|}{p}\sum_{i=1}^{q}\left|g^{G}\cap P_{i}\right|=\frac{\left|C_{G}\left(g\right)\right|}{p}\cdot\left|g^{G}\cap\left(G\backslash Q\right)\right|=\begin{cases}
0 & g\in Q\backslash\left\{ e\right\} \\
q & g\notin Q
\end{cases}
\]
but since $P_{i}$ are all conjugate we have $\mathcal{S}_{P_{i}}=\mathcal{S}_{P_{1}}$
for all $i$. Denoting $P=P_{1}$, we have by the above and \eqref{eq:character_index}

\[
\mathcal{S}_{P}\left(g\right)=\begin{cases}
q & g=e\\
0 & g\in Q\backslash e\\
1 & g\notin Q
\end{cases}
\]
and we find that 
\[
\left(p\cdot\mathcal{S}_{P}+\mathcal{S}_{Q}\right)\left(g\right)=\left(p\cdot\mathcal{S}_{G}+\mathcal{S}_{1}\right)\left(g\right)=\begin{cases}
pq+p & g=e\\
p & g\neq e
\end{cases}
\]
which gives us the unbalanced pair 
\[
X=\underbrace{\nicefrac{G}{P}\cup\ldots\cup\nicefrac{G}{P}}_{p}\cup\,\nicefrac{G}{Q}\qquad\mathrm{and}\qquad Y=\underbrace{\mathbf{1}\vphantom{\nicefrac{G}{P}}\cup\ldots\cup\mathbf{1}}_{p}\cup\, G\,.
\]
This pair was discovered and used for constructing isospectral surfaces
by Hillairet \cite{hillairet2008spectral}.

\subsubsection{\label{sub:Example-dihedral}Example - dihedral groups}

A nice family of groups of the form $\nicefrac{\mathbb{Z}}{q\mathbb{Z}}\rtimes\nicefrac{\mathbb{Z}}{p\mathbb{Z}}$
is formed by the dihedral groups of order $2q$, where $q$ is an
odd prime. $D_{q}=\left\langle \sigma,\tau\,\middle|\,\sigma^{q},\tau^{2},\left(\sigma\tau\right)^{2}\right\rangle $
acts by symmetries on the regular $q$-gon (say, with Neumann boundary
conditions). In this case, the unbalanced pair we obtained above is
$X=\nicefrac{D_{q}}{\left\langle \tau\right\rangle }\cup\nicefrac{D_{q}}{\left\langle \tau\right\rangle }\cup\nicefrac{D_{q}}{\left\langle \sigma\right\rangle }$,
$Y=\mathbf{1}\cup\mathbf{1}\cup D_{q}$, and we obtain for each $q$
an isospectral pair consisting of six orbifolds, five of which are
planar domains with Neumann boundary conditions, and the sixth (the
quotient by $\left\langle \sigma\right\rangle $) a $\frac{2\pi}{q}$-cone.
Figure \ref{fig:D5_quotients} shows the case $q=5$.

\FigBesBeg 
\begin{figure}[h]
\begin{centering}
\includegraphics[scale=1.1]{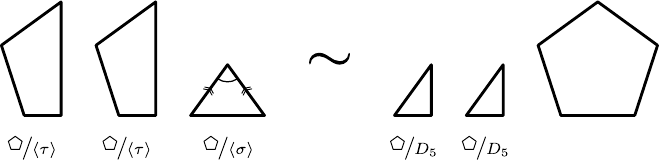}
\par\end{centering}

\caption{\label{fig:D5_quotients}An isospectral pair obtained from the action
of $D_{5}\cong\nicefrac{\mathbb{Z}}{5\mathbb{Z}}\rtimes\nicefrac{\mathbb{Z}}{2\mathbb{Z}}$
on a regular pentagon. All boundary conditions are Neumann.}
 
\end{figure}

\subsection{Non-cyclic groups}

A group $H$ is said to be \emph{involved} in a group $G$ if there
exist some $L\trianglelefteq K\leq G$ such that $\nicefrac{K}{L}\cong H$.
\begin{prop}
\label{pro:involvement}If a group $H$ which has an unbalanced pair
is involved in $G$, then $G$ has an unbalanced pair.\end{prop}
\begin{proof}
It is enough to assume that $H$ is either a subgroup or a quotient
of $G$. Assume first that $H\leq G$. If $X,Y$ is an unbalanced
pair of $H$-sets, the induced $G$-sets $G\times_{H}X$ and $G\times_{H}Y$
(see Section \ref{sub:Tensor-product-of-Gsets}) form an unbalanced
pair as well:
\begin{itemize}
\item They are linearly equivalent: we have natural isomorphisms
\begin{multline*}
\Hom_{\mathbb{C}G}\left(\mathbb{C}\left[G\times_{H}X\right],\_\right)\cong\Hom_{G}\left(G\times_{H}X,\_\right)\\
\cong\Hom_{H}\left(X,\_\right)\cong\Hom_{\mathbb{C}H}\left(\mathbb{C}\left[X\right],\_\right)
\end{multline*}
where the first and last isomorphisms are by \eqref{eq:free-CG-module},
and the middle one is by \eqref{eq:extension_adjointness}. Since
$\mathbb{C}\left[X\right]\cong\mathbb{C}\left[Y\right]$ as $\mathbb{C}H$-modules,
we obtain that $\mathbb{C}\left[G\times_{H}X\right]\cong\mathbb{C}\left[G\times_{H}Y\right]$
as $\mathbb{C}G$-modules. 
\item The sizes of orbits in $G\times_{H}X$ are the sizes of orbits in
$X$ multiplied by $\left[G:H\right]$, since if $X\cong\bigcup\nicefrac{H}{H_{i}}$
is a decomposition of $X$ into $H$-orbits then 
\begin{multline*}
G\times_{H}X\cong G\times_{H}\left(\bigcup\nicefrac{H}{H_{i}}\right)\cong\bigcup G\times_{H}\nicefrac{H}{H_{i}}\\
\cong\bigcup G\times_{H}\times H\times_{H_{i}}\mathbf{1}\cong\bigcup G\times_{H_{i}}\mathbf{1}\cong\bigcup\nicefrac{G}{H_{i}}
\end{multline*}
is a decomposition of $G\times_{H}X$ into $G$-orbits.
\end{itemize}
Assume now that $\pi:G\twoheadrightarrow H$ is an epimorphism. An
$H$-set $X$ has a $G$-set structure by $gx=\pi\left(g\right)x$,
and an unbalanced pair of $H$-sets $X,Y$ is also an unbalanced pair
of $G$-sets: since $G$ realizes the same permutations in $\mathrm{Sym}\left(X\right)$
as does $H$, a linear $H$-equivariant isomorphism $\mathbb{C}\left[X\right]\cong\mathbb{C}\left[Y\right]$
is also $G$-equivariant, and the orbits in $X$ as a $G$-set and
as an $H$-set are the same.\end{proof}
\begin{rem*}
If $G$ acts on a manifold $M$, and $X$ is an $H$-set for some
$H\leq G$, then we have
\[
M\times_{G}\left(G\times_{H}X\right)\cong\left(M\times_{G}G\right)\times_{H}X\cong M\times_{H}X
\]
i.e., the induced $G$-set gives the same manifold as does the original
$H$-set.\end{rem*}
\begin{thm}
\label{thm:non-cyclic-pairs}Every non-cyclic finite group has an
unbalanced pair.\end{thm}
\begin{proof}
Assume that $G$ is finite non-cyclic. If some $p$-Sylow group $P\leq G$
is not cyclic (in particular, if $G$ is abelian), then $\nicefrac{P}{\Phi\left(P\right)}$
contains $\nicefrac{\mathbb{Z}}{p\mathbb{Z}}\times\nicefrac{\mathbb{Z}}{p\mathbb{Z}}$
(here $\Phi\left(P\right)$ is the Frattini subgroup of $P$) and
we are done by Proposition \ref{pro:involvement} and Section \ref{sub:Case-I}.
Zassenhaus classified the groups whose Sylow subgroups are all cyclic
(\cite{hall1976theory}, 9.4.3). They are of the form 
\[
G_{m,n,r}=\left\langle a,b\,\middle|\, a^{m}=b^{n}=e,a^{b}=a^{r}\right\rangle =\nicefrac{\mathbb{Z}}{m\mathbb{Z}}\rtimes_{\vartheta_{r}}\nicefrac{\mathbb{Z}}{n\mathbb{Z}}
\]
for $m,n,r$ satisfying $\left(m,n\left(r-1\right)\right)=1$ (here
$\vartheta_{r}\left(1\right)\left(1\right)=r$, and $r^{n}\equiv1\left(\mathrm{mod}\, m\right)$
is implied to make $\vartheta_{r}$ a homomorphism). Since $0\times\ker\vartheta_{r}\leq Z\left(G\right)$,
and the quotient $\nicefrac{G}{Z\left(G\right)}$ is never cyclic
for nonabelian $G$, we can assume (by Proposition \ref{pro:involvement})
that $\vartheta_{r}$ is injective. We can also assume that $n$ is
prime, for otherwise for any nontrivial factor $k$ of $n$ we have
a proper subgroup $\left\langle a,b^{k}\right\rangle =\nicefrac{\mathbb{Z}}{m\mathbb{Z}}\rtimes_{\vartheta_{r^{k}}}\nicefrac{\mathbb{Z}}{\frac{n}{k}\mathbb{Z}}$
which is non-cyclic by the injectivity of $\vartheta_{r}$. We can
further assume that $m$ is prime. Otherwise, we can pick some prime
$q$ dividing $m$, and consider $\left\langle a^{\nicefrac{m}{q}},b\right\rangle $.
It is cyclic only if $\vartheta_{r}$ fixes $a^{\nicefrac{m}{q}}$,
i.e. $a^{\nicefrac{rm}{q}}=a^{\nicefrac{m}{q}}$, so that $m\mid\frac{m}{q}\left(r-1\right)$,
which is impossible since $\left(m,n\left(r-1\right)\right)=1$. Thus,
by Section \ref{sub:Case-II} we are done.
\end{proof}
Since unbalanced $G$-sets are in particular non-isomorphic, this
together with Proposition \ref{pro:Finite-cyclic-groups} give the
following:
\begin{cor}
\label{cor:Burnside-injective}For a finite group $G$, the map $\Omega\left(G\right)\rightarrow R\left(G\right)$
which takes a $G$-set $X$ to the representation $\mathbb{C}\left[X\right]$
is injective iff $G$ is cyclic.
\end{cor}
Theorems \ref{thm:isospectral_nonisom} and \ref{thm:non-cyclic-pairs}
together give the results announced in Section \ref{sec:Intro}:
\begin{cor}
\label{cor:non-cyclic-isospectral}If a finite non-cyclic group $G$
acts faithfully on a compact connected Riemannian manifold $M$, then
there exist $G$-sets $X,Y$ such that $M\times_{G}X$ and $M\times_{G}Y$
are isospectral and non-isometric.
\end{cor}
from which follows:
\begin{cor}
\label{cor:non-cyclic-quotient}If $M$ is a compact connected Riemannian
manifold (or orbifold) such that $\pi_{1}\left(M\right)$ has a finite
non-cyclic quotient, then $M$ has isospectral non-isometric covers.\end{cor}
\begin{proof}
Let $\widetilde{M}$ be the universal cover of $M$, and $N$ a normal
subgroup in $\pi_{1}\left(M\right)$ such that $G=\nicefrac{\pi_{1}\left(M\right)}{N}$
is finite non-cyclic. $\widehat{M}=\nicefrac{\widetilde{M}}{N}$ is
a finite cover of $M$ and thus compact, and $G$ acts on it faithfully,
with $\nicefrac{\widehat{M}}{G}=M$. By the previous corollary there
exist isospectral non-isometric unions of quotients of $\widehat{M}$
by subgroups of $G$, and these are covers of $M$.
\end{proof}

\section{\label{sec:Computation}Computation}

Here we show how to compute, using GAP \cite{GAP4}, a basis for $\mathcal{L}\left(G\right)$,
the ideal of linearly trivial elements in the Burnside ring $\Omega\left(G\right)$,
which correspond to reduced pairs of linearly equivalent $G$-sets.
We then consider an action of $G$ and compute the isospectral pairs
which correspond to this basis and action.

We take $G=D_{6}$ (see Section \ref{sub:Example-dihedral}), and
choose a set of representatives $\left\{ H_{i}\right\} $ for the
conjugacy classes of subgroups of $G$ (so that $\left\{ \nicefrac{G}{H_{i}}\right\} $
is a $\mathbb{Z}$-basis of $\Omega\left(G\right)$). We then compute
the corresponding quasiregular characters $c_{i}=\mathcal{S}_{H_{i}}$,
which are the images of this basis under the map $\Omega\left(G\right)\rightarrow R\left(G\right)$.
Finally, we compute a basis for $\mathcal{L}\left(G\right)$, the
kernel of this map, and apply the LLL algorithm to this basis in order
to possibly obtain a sparser one.
\begin{align*}
\mathtt{gap>}\: & \mathtt{G:=DihedralGroup(12);;}\\
\mathtt{gap>}\: & \mathtt{H:=List(ConjugacyClassesSubgroups(G),\: Representative);;}\\
\mathtt{gap>}\: & \mathtt{c:=List(H,\: h\;-\!>\: List(PermutationCharacter(G,h)));;}\\
\mathtt{gap>}\: & \mathtt{LLLReducedBasis(NullspaceIntMat(c)).basis;}\\
 & \mathtt{[[0,1,0,-1,0,0,-1,0,1,0],\:[1,-1,0,-1,0,0,0,-1,0,2],}\\
 & \mathtt{[-1,1,0,1,1,0,-1,0,-1,0],\:[-1,1,1,1,0,-2,0,0,0,0]]}
\end{align*}
For example, the first element in the basis we obtained tells us that
$\nicefrac{G}{H_{2}}-\nicefrac{G}{H_{4}}-\nicefrac{G}{H_{7}}+\nicefrac{G}{H_{9}}$
vanishes in $R\left(G\right)$, i.e., that $\nicefrac{G}{H_{2}}\cup\nicefrac{G}{H_{9}}$
is linearly equivalent to $\nicefrac{G}{H_{4}}\cup\nicefrac{G}{H_{7}}$.
One has to explore the output of $\mathtt{ConjugacyClassesSubgroups(G)}$
to find out which subgroups these exactly are, or alternatively, to
construct $H_{i}$ oneself (in this case, for example, $H_{2}$ belongs
to the conjugacy class of $\left\langle \tau\right\rangle $). The
first line in Table \ref{tab:hex_quotients} presents representatives
$H_{i}$ for the classes returned by $\mathtt{ConjugacyClassesSubgroups(G)}$,
and the bottom four lines of the table show the basis that was calculated
for $\mathcal{L}\left(D_{6}\right)$ above. One may check that pairs
$\mathrm{II}$, $\mathrm{III}$ and $\mathrm{IV}$ are unbalanced.

\begin{table}[h]
\begin{centering}
\begin{tabular}{>{\centering}m{15pt}>{\centering}m{25pt}>{\centering}m{15pt}>{\centering}m{22pt}>{\centering}m{20pt}>{\centering}m{15pt}>{\centering}m{22pt}>{\centering}m{22pt}>{\centering}m{10pt}>{\centering}m{25pt}>{\centering}m{22pt}}
\toprule 
{\large \strut }$H_{i}$ & $\left\langle 1\right\rangle $ & $\left\langle \tau\right\rangle $ & $\left\langle \sigma^{3}\right\rangle $ & $\left\langle \tau\sigma\right\rangle $ & $\left\langle \sigma^{2}\right\rangle $ & ${\scriptstyle \left\langle \tau,\tau\sigma^{3}\right\rangle }$ & ${\scriptstyle \left\langle \tau,\tau\sigma^{2}\right\rangle }$ & $\left\langle \sigma\right\rangle $ & ${\scriptstyle \left\langle \tau\sigma,\tau\sigma^{3}\right\rangle }$ & ${\scriptstyle \left\langle \tau,\tau\sigma\right\rangle }$\tabularnewline
\midrule 
{\large \strut }$\nicefrac{\hexagon}{H}_{i}$ & \includegraphics{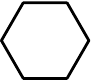} & ~\includegraphics{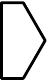} & \includegraphics{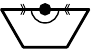} & \includegraphics{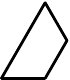} & \includegraphics{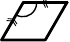} & ~\includegraphics{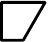} & \includegraphics{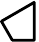} & \includegraphics{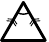} & \includegraphics{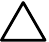} & \includegraphics{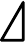}\tabularnewline
\midrule 
$\mathrm{I}$ & 0 & 1 & 0 & -1 & 0 & 0 & -1 & 0 & 1 & 0\tabularnewline
\midrule 
$\mathrm{II}$ & 1 & -1 & 0 & -1 & 0 & 0 & 0 & -1 & 0 & 2\tabularnewline
\midrule 
$\mathrm{III}$ & -1 & 1 & 0 & 1 & 1 & 0 & -1 & 0 & -1 & 0\tabularnewline
\midrule 
$\mathrm{VI}$ & -1 & 1 & 1 & 1 & 0 & -2 & 0 & 0 & 0 & 0\tabularnewline
\bottomrule
\end{tabular}
\par\end{centering}

\caption{\label{tab:hex_quotients}Representatives for the conjugacy classes
of subgroups in $D_{6}$, displayed with the corresponding quotients
of the hexagon, and a basis for $\mathcal{L}\left(D_{6}\right)=\ker\left(\Omega\left(D_{6}\right)\rightarrow R\left(D_{6}\right)\right)$.}
\end{table}

Given an action of $G$ on a manifold $M$, every difference of $G$-sets
$X-Y\in\mathcal{L}\left(G\right)$ gives rise to an isospectral pair,
namely $M\times_{G}X$, $M\times_{G}Y$. We consider the standard
action of $D_{6}$ on the regular hexagon, which we denote by $\hexagon$.
The second line in Table \ref{tab:hex_quotients} shows the quotients
$\nicefrac{\hexagon}{H_{i}}$ corresponding to the subgroups $H_{i}\leq D_{6}$
in the topmost line, and we see that in this case there are no isometric
quotients arising from non-isomorphic $G$-sets. The isospectral pairs
corresponding to the basis we obtained for $\mathcal{L}\left(D_{6}\right)$
are shown in Table \ref{tab:hex_pairs}. 

\begin{table}[H]
\begin{centering}
\begin{tabular}{>{\centering}m{35pt}>{\centering}m{25pt}>{\centering}m{20pt}>{\centering}m{20pt}>{\centering}m{10pt}>{\centering}m{25pt}>{\centering}m{20pt}>{\centering}m{20pt}}
\toprule 
\addlinespace[4pt]
$\mathrm{I}$ &  & \includegraphics{hex_t} & \includegraphics{hex_ts_ts3} & {\Large $\sim$} & \includegraphics{hex_ts} & \includegraphics{hex_t_ts2} & \tabularnewline
\midrule
\addlinespace[4pt]
$\mathrm{II}$ & \centering{}\includegraphics{hex1} & \includegraphics{hex_t_ts} & \includegraphics{hex_t_ts} & {\Large $\sim$} & \includegraphics{hex_t} & \includegraphics{hex_ts} & \includegraphics{hex_s}\tabularnewline
\midrule
\addlinespace[4pt]
$\mathrm{III}$ & \centering{}\includegraphics{hex_t} & \includegraphics{hex_ts} & \includegraphics{hex_s2} & {\Large $\sim$} & \includegraphics{hex1} & \includegraphics{hex_t_ts2} & \includegraphics{hex_ts_ts3}\tabularnewline
\midrule
\addlinespace[4pt]
$\mathrm{IV}$ & \centering{}\includegraphics{hex_t} & \includegraphics{hex_s3} & \includegraphics{hex_ts} & {\Large $\sim$} & \includegraphics{hex1} & \includegraphics{hex_t_ts3} & \includegraphics{hex_t_ts3}\tabularnewline
\midrule
\midrule 
\addlinespace[4pt]
$\mathrm{I}-\mathrm{III}$ & \centering{}\includegraphics{hex1} & \includegraphics{hex_ts_ts3} & \includegraphics{hex_ts_ts3} & {\Large $\sim$} & \includegraphics{hex_ts} & \includegraphics{hex_ts} & \includegraphics{hex_s2}\tabularnewline
\bottomrule
\end{tabular}
\par\end{centering}

\caption{\label{tab:hex_pairs}The isospectral pairs corresponding to the basis
for $\mathcal{L}\left(D_{6}\right)$ described in Table \ref{tab:hex_quotients},
and an example of an element obtained as a combination of these.}
\end{table}

All isospectral pairs which arise from linear equivalences between
$D_{6}$-sets are spanned by these four, as explained in Section \ref{sub:Rings}.
The bottom line in Table \ref{tab:hex_pairs} demonstrates such a
pair (corresponding to the element $\mathrm{I}-\mathrm{III}$). We
remark that the pair corresponding to $\mathrm{I}$ is a hexagonal
analogue of Chapman's two piece band \cite{chapman1995drums} - such
analogues exist for every $n$ (but for odd $n$ the isospectral pair
obtained is also isometric).

\section{\label{sec:Generalizations}Generalizations}

The isospectrality technique this paper describes (and thus Sunada's
technique as well) has actually little to do with spectral geometry,
since no property of the Laplace operator is used apart from being
linear and commuting with isometries. For any linear operator $F$
(on function spaces or other bundles, over manifolds or general spaces),
these methods produce $F$-isospectral objects, given an action of
a group which commutes with $F$.

However, it seems that in much more general settings, when a group
action is studied, Sunada pairs are worth looking at. The most famous
examples are Galois theory, giving Gassmann's construction of arithmetically
equivalent number fields \cite{gassmann1926bemerkungen}, and Riemannian
coverings, giving Sunada's isospectral construction; but Sunada pairs
were also studied in the context of Lie groups \cite{deturck1989isospectral},
ergodic systems \cite{lemanczyk2002relative}, dessin d'enfants \cite{merling2010gassmann},
the spectrum of discrete graphs \cite{brooks1996some} and metric
ones \cite{shapira2006quantum}, the Ihara zeta function of graphs
\cite{stark2000zeta}, and the Witten zeta function of a Lie group
\cite{larsen2004determining}.

Sunada pairs in $G$ correspond to linearly equivalent transitive
$G$-sets, and we have seen that in the context of Riemannian coverings
Sunada's technique generalizes to non-transitive $G$-sets as well.
We achieved this by considering the quotient $\nicefrac{M}{H}$ as
the tensor product with the transitive $G$-set $\nicefrac{G}{H}$,
i.e., by noting that $\nicefrac{M}{H}\cong M\times_{G}\nicefrac{G}{H}$,
and then studying $M\times_{G}X$ for a general $G$-set $X$%
\footnote{Alternatively, for $X=\bigcup\nicefrac{G}{H_{i}}$ we studied the
disjoint union of quotients $\bigcup\nicefrac{M}{H_{i}}$.%
}. It is natural to ask whether other applications of Sunada pairs
can be generalized in an analogous way. Of particular interest are
unbalanced pairs, which do not exist in the transitive case (see Remark
\ref{rem:No-unbalanced-Sunada}). In the settings of Riemannian manifolds
they allowed us to deduce non-isometry, and one may hope that they
play interesting roles in other situations.

\paragraph*{Acknowledgements}

I am grateful to my advisor, Alex Lubotzky, for his guidance, and
to Peter Doyle and Carolyn Gordon for their encouragement. I have
had fruitful discussions, and received valuable comments from Menny
Aka, Emmanuel Farjoun, Luc Hillairet, Asaf Horev, Doron Puder, and
Ron Rosenthal, and I thank them heartily.

\bibliographystyle{plain}
\bibliography{/home/math/Roaming/Math/BibTex,/home/ori/Math/BibTex}

\end{document}